\documentclass[UTF8]{article}
\usepackage{amssymb} 
\usepackage{amsmath} 
\usepackage{amsthm}
\usepackage[utf8]{inputenc} 
\usepackage[english]{babel} 
\usepackage[pdftex]{graphics}
\usepackage{graphicx}
\usepackage[all]{xy}
\usepackage{enumerate}
\usepackage{tikz-cd}
\usepackage{mathtools}
\numberwithin{equation}{section}
\theoremstyle{Theorem}
\newtheorem{thm}{Theorem}[section]
\newtheorem{lemma}[thm]{Lemma}
\newtheorem{prop}[thm]{Proposition}
\newtheorem{coro}[thm]{Corollary}
\theoremstyle{definition}
\newtheorem{dfn}[thm]{Definition}
\newtheorem{exm}[thm]{Example}
\newtheorem{rmk}[thm]{Remark}

\newtheorem{innercustomthm}{Theorem}
\newenvironment{customthm}[1]{\renewcommand\theinnercustomthm{#1}\innercustomthm}{\endinnercustomthm}

\usepackage[paperwidth=17cm, paperheight=22.5cm, bottom=2.5cm, right=2.5cm]{geometry}
\usepackage[T1]{fontenc}	
\usepackage[linktocpage=true]{hyperref}
\usepackage{tocloft} 
\usepackage{authblk}
\usepackage{appendix}

\usepackage{natbib,hyperref}

\cftsetindents{section}{0em}{2em}

\newcommand\norm[1]{\lVert#1\rVert}
\newcommand\grau[1]{\lvert#1\rvert}

\def\R{\mathbb{R}}
\def\g{\mathfrak{g}}

\def\Z{\mathbb{Z}}
\def\R{\mathbb{R}}
\def\N{\mathbb{N}}

\def\tsigma{\tilde{\sigma}}

\makeatletter
\def\thanks#1{\protected@xdef\@thanks{\@thanks
        \protect\footnotetext{#1}}}
\makeatother

\title{\textbf{Equivariant deformation problems and homotopy operators}}

\author{Sebastián Daza\thanks{scdaza.95@gmail.com}}
\author{Jo\~ao Nuno Mestre\thanks{jnmestre@proton.me}}
\affil{CMUC, University of Coimbra, Department of Mathematics, Portugal.}

\date{}
\begin{document}
\maketitle
\begin{abstract}
We use homotopy operators for the $L_\infty$-algebra associated with an equivariant deformation problem in order to describe a smooth parametrization of the space of structures around a given one.
 Along the way we give new algebraic and explicit proofs of rigidity and unobstructedness theorems.
\end{abstract}
\tableofcontents
\section{Introduction}
When studying an algebraic or geometric structure (e.g. associative algebras, Lie algebras, complex structures), it is fruitful to understand how it changes under small variations. In other words, the goal is to understand a neighbourhood of a given structure inside the space of such structures up to equivalence - roughly a \textit{moduli space}.

By considering \textit{deformations}, i.e. paths in the space of structures starting at a given one, we can approximate infinitesimally such neighbourhoods by the spaces of tangent vectors to deformations. 
The infinitesimal consequences of the equations that the structures under observation must satisfy (e.g associativity, the Jacobi identity, Cauchy-Riemann equations) lead to the construction of a deformation cochain complex, such that tangent vectors to deformations represent deformation cocycles. 

This approach has been standard since the classical works on the deformation theories of complex structures by Fr\"{o}licher--Nijenhuis \cite{FN} and Kodaira--Spencer \cite{KodairaSpencerI, KodairaSpencerIII}, of associative algebras by Gerstenhaber \cite{gerstI1964}, and of Lie algebras by Nijenhuis--Richardson \cite{NijRicI}, among many others since then. 

All these works found descriptions of appropriate deformation complexes. It was also shown in \cite{Kur1, GerstBra, NijRicI}, and stressed in \cite{NijRicII}, that the complexes carried additional structure relevant to describing moduli spaces, in the form of compatible Lie brackets (making the deformation complex into a differential graded Lie algebra). This structure is useful, for example, in order to understand which deformation cocycles can arise as tangent vectors to actual deformations (e.g. in \cite{Kur1}). 

From the early observations of Nijenhuis--Richardson, a guiding principle arose, postulated by Deligne \cite{Deligne}, Drinfeld \cite{letter}, and in an equivalent formulation by Schlessinger--Stasheff \cite{Schlessinger_Stasheff}: Any reasonable deformation problem in characteristic zero is controlled by a differentiable graded Lie algebra. This principle has been recently turned into a theorem \cite{pridham,lurie}, stated as an equivalence between the $\infty$-categories of formal moduli problems and of differential graded Lie algebras. Nonetheless, there is no general procedure to explicitly describe the Lie brackets for any deformation cochain complex, although there are different methods available for several classes of examples.

In this paper we are interested in studying deformation problems for which the space of structures can be described as the space $\sigma^{-1}(0)$ of zeros of a section $\sigma\in \Gamma(E)$ of a vector bundle $\pi:E\rightarrow M$. 
In fact, we will consider the extra structure 
\[
				\begin{tikzcd}
					 E\curvearrowleft G \arrow[d]  \arrow[r,"\Phi"] & F\arrow[dl]\\
					\arrow[ u, bend left, dashed, "\sigma"]  M \curvearrowleft G
				\end{tikzcd}
\]
consisting of suitable Lie group actions of $G$ on $M$ and on $E$, an equivariant section $\sigma\in \Gamma_G(E)$  and a vector bundle map $\Phi$  satisfying $\Phi\circ \sigma=0$.  We call this structure an \textit{equivariant deformation problem}. 

Consider $\sigma^{-1}(0)\subset M$ as a topological space with the subspace topology.
Intuitively, $M$ describes the space of `almost' structures, while $\sigma^{-1}(0)$ describes the space of actual structures. For example, $\sigma$ being an associator or Jacobiator, leads to $\sigma^{-1}(0)$ being the spaces of associative or Lie algebra structures on a fixed vector space. The action of $G$ identifies equivalent structures; the role of $F$ and $\Phi$ will become clear later. We are interested in the following:

 \ \
 
 \textbf{Main question:} For a given solution $x_0\in \sigma^{-1}(0)$,
 
  is there an open $x_0\in U\cap\sigma^{-1}(0)$ that admits a smooth structure?
 
\ \
 
If such an open exists we say that the equivariant deformation problem is \textit{integrable} at  $x_0$. On the other hand, the structure defining an equivariant deformation problem implies that $G(x_0)\subset \sigma^{-1}(0)$, with $ G(x_0)$ the orbit. Given that $G(x_0)$ has a smooth structure, another natural question is whether the orbit is open in $\sigma^{-1}(0)$ around $x_0$, i.e. $G(x_0)\cap U= \sigma^{-1}(0)\cap U$. In this case we call $x_0$ \textit{rigid}. Now, consider the sequence 
\begin{equation}\label{ses}
	\g\overset{d_{e}m_{x_0}}{\rightarrow} T_{x_0} M \overset{d_{x_0}^v\sigma}{\rightarrow} E_{x_0}\overset{\Phi_{x_0}}{\rightarrow} F_{x_0},
\end{equation}
with $m_{x_0}: G\rightarrow M$ the action at $x_0$ and $d^v_{x_0} \sigma$ the vertical derivative.  In \cite{crainic2014survey} the authors proved that the exactness of the sequence at $T_{x_0}M$ and $E_{x_0}$ implies rigidity of $x_0$, and integrability at $x_0$, respectively (see Propositions $4.3$ and $4.4$). Their proof relies, essentially, on the inverse function theorem: it is a clever concatenation of the constant rank theorem and transversality. In fact, their results can be stated as:                                                                                                                                                                                                                                                                                                                                                                                                                                                                                                                                                                                                                                                                                                                                                                                                                                                                                                                                                                                                                                                                                                                                                                                                                                                                                                                                                                                                                                                                                                                                                                                                                                                                                                                                                                                                                                                                                                                                                                                                                                                                                                                                                                                                                                                                                                                                                                                                                                                                                                                                                                                                                                                                                                                                                                                                                                                                                                                                                                                                                                                                                                                                                                                                                                                                                                                                                                                                                                                                                                                                                                                                                                                                                                                                                                                                                                                                                                                                                                                                                                                                                                                                                                                                                                                                                                                                                                                                                                                                                                                                                                                                                                                                                                                                                                                                                                                                                                                                                                                                                                                                                                                                                                                                                                                                                                                                                                                                                                                                                                                                                                                                                                                                                                                                                                                                                                                                                                                                                                                                                                                                                                                                                                                                                                                                                                                                                                                                                                                                                                                                                                                                                                                                                                                                                                                                                                                                                                                                                                                                                                                                                                                                                                                                                                                                                                                                                                                                                                                                                                                                                                                                                                                                                                                                                                                                                                                                                                                                                                                                                                                                                                                                                                                                                                                                                                                                                                                                                                                                                                                                                                                                                                                                                                                                                                                                                                                                                                                                                                                                                                                                                                                                                                                                                                                                                                                                                                                                                                                                                                                                                                                                                                                                                                                                                                                                                                                                                                                                                                                                                                                                                                                                                                                                                                                                                                                                                                                                                                                                                                                                                                                                                                                                                                                                                                                                                                                                                                                                                                                                                                                                                                                                                                                                                                                                                                                                                                                                                                                                                                                                                                                                                                                                                                                                                                                                                                                                                                                                                                                                                                                                                                                                                                                                                                                                                                                                                                                                                                                                                                                                                                                                                                                                                                                                                                                                                                                                                                                                                                                                                                                                                                                                                                                                                                                                                                                                                                                                                                                                                                                                                                                                                                                                                                                                                                                                                                                                                                                                                                                                                                                                                                                                                                                                                                                                                                                                                                                                                                                                                                                                                                                                                                                                                                                                                                                                                                                                                                                                                                                                                                                                                                                                                                                                                                                                                                                                                                                                                                                                                                                                                                                                                                                                                                                                                                                                                                                                                                                                                                                                                                                                                                                                                                                                                                                                                                                                                                                                                                                                                                                                                                                                                                                                                                                                                                                                                                                                                                                                                                                                                                                                                                                                                                                                                                                                                                                                                                                                                                                                                                                                                                                                                                                                                                                                                                                                                                                                                                                                                                                                                                                                                                                                                                                                                                                                                                                                                                                                                                                                                                                                                                                                                                                                                                                                                                                                                                                                                                                                                                                                                                                                                                                                                                                                                                                                                                                                                                                                                                                                                                                                                                                                                                                                                                                                                                                                                                                                                                                                                                                                                                                                                                                                                                                                                                                                                                                                                                                                                                                                                                                                                                                                                                                                                                                                                                                                                                                                                                                                                                                                                                                                                                                                                                                                                                                                                                                                                                                                                                                                                                                                                                                                                                                                                                                                                                                                                                                                                                                                                                                                                                                                                                                                                                                                                                                                                                                                                                                                                                                                                                                                                                                                                                                                                                                                                                                                                                                                                                                                                                                                                                                                                                                                                                                                                                                                                                                                                                                                                                                                                                                                                                                                                                                                                                                                                                                                                                                                                                                                                                                                                                                                                                                                                                                                                                                                                                                                                                                                                                                                                                                                                                                                                                                                                                                                                                                                                                                                                                                                                                                                                                                                                                                                                                                                                                                                                                                                                                                                                                                                                                                                                                                                                                                                                                                                                                                                                                                                                                                                                                                                                                                                                                                                                                                                                                                                                                                                                                                                                                                                                                                                                                                                                                                                                                                                                                                                                                                                                                                                                                                                                                                                                                                                                                                                                                                                                                                                                                                                                                                                                                                                                                                          
\[
\text{infinitesimal integrability/rigidity} \ \Rightarrow \text{integrability/rigidity}.
\]
These results did not yet use the additional algebraic structure that we expect the deformation complex \ref{ses} to carry. In the PhD thesis \cite{baarsma2019deformations} the author showed that the Taylor expansions of the structures defining the equivariant deformation problem induce a (curved) $L_\infty$-structure $(V,\ell)$ on the sequence \ref{ses}  \cite[Theorem 5.2.4]{baarsma2019deformations}. In fact, the equation 
\[
	\sum_{k\geq 0} \dfrac{1}{k!} d^k\sigma(0)(v,\overset{k}{\cdots}, v)=0,
\]
corresponds to the so called \textit{Maurer-Cartan equation} 
\[
	\sum_{k\geq 0} \dfrac{1}{k!} \ell_k(v, \overset{k}{\cdots}, v)=0.
\]
We call the solutions of this equation \textit{Maurer-Cartan elements}. Accordingly, when $\sigma$ is analytic around $x_0$, we get a local correspondence between $\sigma^{-1}(0)$ and Maurer-Cartan elements \cite[Theorem 5.2.5]{baarsma2019deformations}. Given that the sequence \ref{ses} corresponds to 
\[
	V_{-1} \overset{\ell_1}{\rightarrow} V_0 \overset{\ell_1 }{\rightarrow} V_1 \overset{\ell_1}{\rightarrow} V_2,
\]
and the exactness of \ref{ses} at $T_{x_0}M$ and $E_{x_0}$ amounts to the vanishing of the cohomologies $H^0(V,\ell)$ and $H^1(V,\ell)$, then
\begin{align*}
	\text{infinitesimal rigidity} \ &\Leftrightarrow \ H^0(V,\ell)=0,\\
	\text{infinitesimal integrability} \ &\Leftrightarrow \  H^1(V,\ell)=0.
\end{align*}
As long as $V$ is finite dimensional, the vanishing $H^i(V,\ell)=0$ is equivalent to the existence of linear maps $h_1: V_{i}\rightarrow V_{i-1}$ and $h_2: V_{i+1}\rightarrow V_{i}$ such that 
\[
	\ell_1\circ h_1+h_2\circ \ell_1= id_{V_i}.
\]
The linear maps $h_1,h_2$ are called \textit{homotopy operators in degree} $i$.
\\

 Our main contributions are to show how to construct, from homotopy operators in degrees $0$ and $1$, smooth structures on $\sigma^{-1}(0)$ around $x_0$. 
 For rigidity we get a parametrization of the orbit in terms of homotopy operators in degree $0$. Additionally, we give a new proof of a rigidity result for Lie algebras, in terms of the parallel transport of a connection defined in terms of homotopy operators. 
 
 Integrability is more subtle. We will use the local correspondence between $\sigma^{-1}(0)$ and Maurer-Cartan elements. In fact, we give a differential geometric proof of the well-known fact that the cohomology $H^1(V,\ell)$ of an $L_ \infty$-algebra controls the existence of formal Maurer-Cartan elements. Indeed, we get a recursive process to construct a formal Maurer-Cartan element $u_t=\sum_{k\geq 0} \frac{u_k}{k!} t^k$, by solving the infinite sequence of cohomological equations
\[
	\ell_1(u_{k+1})= -Obs^k(u_0, \ldots, u_{k}),
\]
where $Obs^k(u_0,\ldots, u_{k})\in H^1(V,\ell)$ are appropriate obstruction classes. We use homotopy operators in degree $1$ to provide the explicit solutions 
\[
	u_{k+1}=- h_1(Obs^k(u_0, \ldots, u_k)).
\]
Finally we show that, provided that the $L_\infty$-algebra is $N$-strict, $u_t$ converges. Using this construction, we describe a smooth structure on $\sigma^{-1}(0)$ around $x_0$. These arguments and constructions, in terms of homotopy operators, can be compared with Remark $4.6$ of \cite{crainic2004perturbation}, which says that the perturbation lemma corresponds to the algebraic version of Newton's iteration method. Indeed, we use a homotopy operator, recursively, to construct a formal solution and then we prove its convergence. 

\subsection*{Outline of the paper}

In section 2 we use homotopy operators to give an alternative, algebraic proof of a rigidity result (Theorem \ref{rigiditysurvey}) for equivariant deformation problems:

\begin{customthm}{1}
Let $x_0\in \sigma^{-1}(0)$. If $x_0$ is infinitesimally rigid then it is rigid.
\end{customthm}

Additionally, we use homotopy operators for the deformation complex of Lie algebras to construct a connection on a tautological bundle of Lie algebras. Using the associated parallel transport we give a new proof of a rigidity result (see Theorem \ref{rigidityPT}).

\begin{customthm}{2}
If $H^1(\mathfrak{g})=0$ then $\mathfrak{g}$ is rigid. 
\end{customthm}

In section 3 we provide an introduction to $L_\infty$-algebras and Maurer-Cartan elements.

In section 4 we give an explicit description of the obstruction classes $Obs^k$, and use them to recursively construct formal Maurer-Cartan elements (see Theorem \ref{formalint}):

\begin{customthm}{3} Let $ (V,\ell)$ be an $L_\infty$-algebra such that 
$H^1(V,\ell)=0$. Then any infinitesimal deformation can be extended to a formal deformation.
\end{customthm}

In section 5 we use the explicit description of formal Maurer-Cartan elements obtained before, together with homotopy operators, in order to obtain an integrability result (see Theorem \ref{intdgla} and Corollary \ref{MaxIntNew}), i.e. smooth parametrizations for $\sigma^{-1}(0)$ around a structure $x_0$.

\begin{customthm}{4}
Let $(M\overset{\sigma}{\rightarrow} E\overset{\Phi}{\rightarrow} F)\curvearrowleft G$ be an analytic deformation problem with $(V,\ell)$ its associated deformation $L_\infty$-algebra. If $(V,\ell)$ is $N$-strict and $H^1(V,\ell)=0$, then $\sigma^{-1}(0)$ is smooth around $x_0$.
\end{customthm}

We collect in Appendix A the combinatorial background material used in the proofs; in Appendix B we compute higher order derivatives of the differential of an $L_\infty$-structure.

\subsection*{Acknowledgements} The authors would like to thank Camilo Arias Abad, Diogo Soares and Ricardo Campos for fruitful discussions related to the work of this paper.

This work was supported by the Fundação
para a Ciência e a Tecnologia (FCT) under the scope of the project UID/00324 – Center for Mathematics of the University of Coimbra. Sebasti\'an Daza acknowledges FCT for support under the Ph.D. Scholarship UI/BD/152072/2021.

\section{Rigidity}\label{Sec-rigidity}
\subsection{Deformations}
Let $\pi_E:E\rightarrow M$ be a vector bundle  with an action $E\curvearrowleft G$ such that the zero section $0_M\cdot G=0_M$ is $G$-invariant. Accordingly, $M$ inherits an action of $G$. Take a $G$-equivariant section $\sigma\in\Gamma_G (E)$. We want to study the space of solutions of $\sigma(x)=0$. Fix $x_0\in M$ such that $\sigma(x_0)=0$ and consider $\sigma^{-1}(0)\subset M$ with the subspace topology.
\begin{dfn}
A smooth path $x_t: I\rightarrow M$ that starts at $x_0$ and satisfies $\sigma(x_t)=0$ for all $t\in I$ is called a \textit{deformation of} $x_0$. 
\end{dfn}
The \textit{vertical derivative} of $\sigma$ at $x_0$ is the linear map $d_{x_0}^v\sigma: T_{x_0}M\rightarrow E_{x_0}$ given by 
\[
	T_{x_0}M\overset{d_{x_0}\sigma} {\rightarrow }T_{0_{x_0}} E \overset{\mathrm{can}}{\cong} T_{x_0} M\oplus E_{x_0} \overset{pr}{\rightarrow} E_{x_0}, 
\]
where the isomorphism is canonically induced by the zero section $0_M$. The following is a standard result in deformation theory.
\begin{prop}\label{Tsigma}
If $x_t$ is a deformation of $x_0$ then $\partial_{t=0} \ x_t\in \ker d_{x_0}^v \sigma$. 
\end{prop}

In this sense we think of $\ker d_{x_0}^v\sigma$ as the model space for the tangent space ``$T_{x_0}\sigma^{-1}(0)$''. In other words, the best situation we can expect is that $\sigma^{-1}(0)$ has locally, around $x_0$, a manifold structure modelled on $\ker d_{x_0}^v\sigma$.

\subsection{Rigidity}
Let $G(x_0)$ be the orbit of $x_0$ under the action $M\curvearrowleft G$. By the $G$-equivariance of $\sigma$ we have that $G(x_0)\subset \sigma^{-1}(0)$. 
\begin{dfn}
A solution $x_0$ is called \textit{rigid} if there exists an open neighbourhood $x_0\in U\subset M$ such that $G(x_0) \cap U=\sigma^{-1}(0)\cap U$.
\end{dfn}
In the next Proposition we show that rigidity, seen as openness of the orbit around $x_0$, is equivalent to another notion of rigidity found commonly in the literature, for example in Proposition $4.3$ of \cite{crainic2014survey}. 
\begin{prop}\label{eqrig}
A solution $x_0$ is rigid if and only if there exists an open set $x_0\in U\subset M$ and smooth map $h: U\rightarrow G$ such that, for every $y\in \sigma^{-1}(0)\cap U$, we have that $y=x_0\cdot h(y)$. 
	\begin{proof}
		$(\Rightarrow)$ Given that orbits of Lie group actions are locally embedded, there exists an open neighbourhood $x_0\in U\subset M$ such that $V':=G(x_0)\cap U$ is an open of $G(x_0)$. 
				Take a relatively compact open set $V$ such that $x_0\in V\subset  \bar{V}\subset V^\prime\subset G(x_0)$. Let $G_{x_0}$ be the isotropy of $G$ at $x_0$. The diffeomorphism $\mu:G/G_{x_0}\cong G(x_0)$ gives us a relatively compact open set $W=\mu^{-1}(V)$ of $W'=\mu^{-1}(V')$ such that $[e]\in W\subset \bar{W}\subset W^\prime \subset G/G_{x_0}$. Given that $G_{x_0}\subset G$ is a closed subgroup then $G\rightarrow G/G_{x_0}$ is a $G_{x_0}$-principal bundle.  
				Shrink $W'$ if necessary and take a section $s\in \Gamma_{W^\prime}( G\rightarrow G/G_{x_0})$. It induces a smooth map $h=s\circ \mu^{-1}: V^\prime \rightarrow G$ such that $y=x_0\cdot h(y)$. The result follows from extending $h: \bar{V}\rightarrow G$ to a smooth map $h: U\subset M\rightarrow G$. Indeed, shrinking $U$ if necessary, since $x_0$ is rigid, we can assume that $G(x_0) \cap U=\sigma^{-1}(0)\cap U$.  \\
		$(\Leftarrow)$ We know that $G(x_0)\subset \sigma^{-1}(0)$. By hypothesis $ \sigma^{-1}(0)\cap U\subset G(x_0)\cap U$ and the result follows. 
	\end{proof}
\end{prop}
Let $m_{x_0}: G\rightarrow M$ be $g\mapsto x_0\cdot g$. By the $G$-equivariance of $\sigma$, the sequence 
	\begin{equation}\label{rigidseq}
		\begin{tikzcd}
			\g \arrow[r," d_e m_{x_0}"]& T_{x_0} M \arrow[r," d_{x_0}^v\sigma"] & E_{x_0}.
		\end{tikzcd}
	\end{equation}
is a cochain complex. 
\begin{dfn}
The solution $x_0$ is called \textit{infinitesimally rigid} if the sequence \ref{rigidseq} is exact at $T_{x_0}M$. 
\end{dfn}
Proposition $4.3$ of \cite{crainic2014survey} says the following:
\begin{thm}\label{rigiditysurvey}
Let $x_0\in \sigma^{-1}(0)$. If $x_0$ is infinitesimally rigid then it is rigid.
\end{thm}
\subsection{Homotopy operators and rigidity}
Let $(V^\bullet, \delta)$ be a cochain complex and $H_\delta^\bullet(V)$ its cohomology. 
\begin{dfn}
The pair of linear maps $h_{k+1}: V^{k+1}\rightarrow V^k$ and $h_k: V^k\rightarrow V^{k-1}$ are called \textit{homotopy operators in degree} $k$ if
\[
	\delta\circ h_k + h_{k+1}\circ \delta= Id.
\]
\end{dfn}
The following is a standard result:
\begin{prop}\label{hom}
Let $V^\bullet$ be finite dimensional. Then $H^k_\delta(V)=0$ if and only if there exists homotopy operators in degree $k$. 
\end{prop}
Now we are going to give an alternative proof of Theorem \ref{rigiditysurvey} using homotopy operators. Even though the idea of the proof is in essence the same, i.e. to use the constant rank theorem, we are going to see how homotopy operators provide an explicit parametrization of the orbit. Moreover, they give a clearer frame on how, and why, the constant rank theorem is used. First of all, since the results we want to prove are local, we suppose $E=M\times V$ for some vector space $V$ and $\sigma: M\rightarrow V$. Hence $d^v_{x_0}\sigma$ corresponds to $d_{x_0}\sigma$. 

\begin{proof}[Alternative proof of Theorem \ref{rigiditysurvey}]
 By hypothesis the sequence \ref{rigidseq} is exact and so by Proposition \ref{hom} there exists homotopy operators $h_2: V \rightarrow T_{x_0}M$  and $h_1: T_{x_0}M\rightarrow \g$ such that 
\begin{equation}\label{homotopyrelation}
	d_em_{x_0}\circ h_1+ h_2\circ d_{x_0}\sigma= Id. 
\end{equation}
Let $\psi_h: \ker d_{x_0}\sigma\rightarrow G(x_0)\subset M$ be given by the composition
\[
	\begin{tikzcd}
			 \ker d_{x_0}\sigma \arrow[r, "h_1"] & \g \arrow[r," \exp"]& G  \arrow[r," m_{x_0}"] & G(x_0)\subset M.
		\end{tikzcd}
\]
Equation \ref{homotopyrelation} implies that $d_0 \psi_h= Id$. Additionally, $\dim G(x_0)= \mathrm{rk}\ d_em_{x_0} = \dim  \ker d_{x_0}\sigma$ and so $\psi_h$ gives a local parametrization of the orbit around $x_0$. Moreover, since $\psi_h$ is an embedding, there exists a complement $T_{x_0} M= \ker d_{x_0}\sigma\oplus C$ and a chart $\varphi: U\subset M \rightarrow \ker d_{x_0}\sigma\oplus C$, with $x_0\mapsto 0$, and such that 
\[
	y\in G(x_0) \Longleftrightarrow \varphi(y)=(v,0)\in \ker d_{x_0}\sigma\oplus C.
\]
Hence $x_0$ is rigid if and only if $\sigma(v,c)=0$ implies that $c=0$. To see this, define $\Phi:  \ker d_{x_0}\sigma\oplus C\rightarrow  \ker d_{x_0}\sigma\oplus V$ by $(v,c)\mapsto (v, \sigma(v,c))$. Given that $\sigma(v,0)=0$ for all $v$ and $\mathrm{rk} \ d_{x_0}\sigma = \dim C$, then $d_{0}\Phi$ has maximal constant rank and the result follows. 
\end{proof}

\begin{rmk}
With respect to the previous structure, the map $h: U\subset M\rightarrow G$ given in Proposition \ref{eqrig} corresponds to 
\[
	h:= \exp\circ h_1\circ (Id-h_2\circ d_{x_0}\sigma)\circ \varphi.
\]
\end{rmk}

\subsection{Rigidity of Lie algebras and parallel transport}
Let $\g$ be a finite dimensional vector space and let $C^k(\g):= Hom(\wedge^k\g,\g)$. Consider the action $C^k(\g)\curvearrowleft GL(\g)$ given by 
		\[
			(\eta\cdot A)(x_1,\ldots, x_k):= A^{-1}\eta(Ax_1, \ldots, Ax_k). 
		\]
A \textit{Lie algebra structure} on $\g$ is an element $\mu\in C^2(\g)$ which is a zero of the Jacobiator $Jac: C^2(\g)\rightarrow C^3(\g)$, given by 
\[
	Jac(\mu)(x,y,z)= \mu(\mu(x,y),z)+\mu(\mu(y,z),x)+\mu(\mu(z,x),y).
\]
Accordingly, Lie algebras can be thought as zeros of the equivariant section  
		\[
				\begin{tikzcd}
					 C^2(\g)\times C^3(\g) \curvearrowleft GL(\g) \arrow[d] \\
					\arrow[ u, bend left, dashed, " Jac"]  C^2(\g) \curvearrowleft GL(\g) 
				\end{tikzcd}
		\]
We denote the space of Lie algebra structures on $\g$ by $Lie(\g):= Jac^{-1}(0)$. 

The trivial vector bundle $\tau_{C^2(\g)}:=C^2(\g)\times \g\rightarrow C^2(\g)$ is equipped with a tautological skew-symmetric bilinear operation $[-,-]:\wedge^2\tau_{C^2(\g)}\to\tau_{C^2(\g)}$, given by
$$[-,-]: C^\infty(C^2(\g), \g)\times C^\infty(C^2(\g), \g) \rightarrow C^\infty(C^2(\g), \g),$$
\[
	[ \alpha, \beta] (\mu)= \mu(\alpha(\mu),\beta(\mu)). 
\]

We call the pair $(\tau_{C^2(\g)}, [-,-])$ the \textit{tautological bundle of} $C^2(\g)$. The restriction of $\tau_{C^2(\g)}$ to the subspace $Lie(\g)\subset C^2(\g)$ is a topological vector subbundle; the restriction of $[-,-]$ makes it into a bundle of Lie algebras.

\begin{dfn}
The \textit{tautological bundle} of $\g$ is the (topological) bundle of Lie algebras $\tau_\g:=(Lie(\g)\times \g\rightarrow Lie(\g), [-,-])$. 
\end{dfn}

Now fix a $\mu_0\in Lie(\g)$ and suppose that we have homotopy operators $h_2^{\mu_0}: C^3(\g)\rightarrow C^2(\g)$ and $h_1^{\mu_0}: C^2(\g)\rightarrow C^1(\g)$ for the sequence
\begin{equation*}
		\begin{tikzcd}
			C^1(\g) \arrow[r," d_e m_{\mu_0}"]& C^2(\g) \arrow[r," d_{\mu_0} Jac"] & C^3(\g).
		\end{tikzcd}
\end{equation*}
Take an open set $\mu_0\in U\subset C^2(\g)$ such that the maps $H_2:U\times C^3(\g)\rightarrow C^2(\g)$ and $H_1:U\times C^2(\g)\rightarrow C^1(\g)$, given by 

\begin{align*}
H_2(\mu,\cdot):=& h_2^{\mu_0}\circ (1-  (d_{\mu-\mu_0} Jac)\circ h_2^{\mu_0})^{-1},\\
H_1( \mu,\cdot) :=& h_1^{\mu_0}\circ (1-  (d_{\mu-\mu_0} m) \circ h_2^{\mu_0})^{-1}
\end{align*}

are well-defined. Denote $h_2^\mu= H_2(\mu, \cdot)$ and $h_1^\mu=H_1(\mu,\cdot)$. By Proposition $3.4$ of \cite{crainic2004perturbation}, whenever $\mu\in Lie(\g)\cap U$, we have that $h_1^\mu$ and $h_2^\mu$ are homotopy operators for 
\begin{equation}\label{laseq2}
		\begin{tikzcd}
			C^1(\g) \arrow[r," d_e m_{\mu}"]& C^2(\g) \arrow[r," d_{\mu} Jac"] & C^3(\g).
		\end{tikzcd}
\end{equation}
Consider the restriction of $\tau_{C^2(\g)}$ to $U$, the bundle $(U\times \g\rightarrow U,[-,-])$. Given that $U\subset C^2(\g)$ is an open subset of a vector space, we can identify $\mathfrak{X}(U)$ with $C^\infty(U, C^2(\g))$. Define the connection $\nabla: \mathfrak{X}(U) \times\Gamma(U\times \g) \rightarrow \Gamma(U\times\g)$ by
\[
	\nabla_X \xi :=  \mathcal{L}_X \xi + H_1(X)(\xi). 
\]
A direct computation shows the following:
\begin{lemma}
For every $\alpha,\beta\in \Gamma( U\times \g)$ and $X\in\mathfrak{X}(U)$ we have that 
\begin{multline}\label{lem}
	\nabla_X [\alpha,\beta]-[\nabla_X \alpha,\beta]-[\alpha,\nabla_X\beta]= \\
	 X(\alpha, \beta)+ H_1(X)( [\alpha,\beta])- [H_1(X)(\alpha),\beta]- [\alpha, H_1(X)(\beta)]. 
\end{multline}
\end{lemma}
\begin{prop}
Let $\mu_t: I\rightarrow U\cap Lie(\g)$ be a deformation of $\mu_0$. Then, for every $\alpha,\beta\in \Gamma(U\times \g)$, we have that
\[
	\nabla_{\partial_t \mu_t} [\alpha,\beta]= [\nabla_{\partial_t \mu_t} \alpha, \beta]+[\alpha, \nabla_{\partial_t \mu_t} \beta].
\]
\begin{proof}
By Proposition \ref{Tsigma} we know that $\partial_t \gamma(t)\in \ker d_{\gamma(t)} Jac$. The fact that $h_1^{\mu_t}$ is an homotopy operator for the sequence \ref{laseq2}, with $\mu=\mu_t$, implies that  
\[
	(d_{e}m_{\mu_t} \circ h_1^{\mu_t})(\partial_t \mu_t)= \partial_t \mu_t. 
\]
But for $A\in C^1(\g)$ we get
\[
	d_e m_{\mu_t} (A)(x,y)= \mu_t(A(x), y)+\mu_t (x, A(y))- A(\mu_t(x,y)),
\]
and then 
\[
	 (\partial_t \mu_t) (x, y)+ h_1^{\mu_t}(\partial_t \mu_t)( \mu_t(x,y))- \mu_t(h_1^{\mu_t}(\partial_t \mu_t )(x),y)- \mu_t(x, h_1^{\mu_t}(\partial_t \mu_t)(y))=0. 
\]
Letting $x=\alpha(\mu_t)$, $y=\beta(\mu_t)$, $X=\partial_t \mu_t$, and recalling that $[-,-](\mu_t)=\mu_t$, we conclude that the right hand side of \ref{lem} vanishes and the result follows.
\end{proof}
\end{prop}
The deformation $\mu_t$ is called \textit{trivial} if $(\g,\mu_t)\cong (\g,\mu_0)$ for all $t$. 
\begin{coro}
If $H^1(\mu_0)=0$ then every (small) deformation $\mu_t$ of $\mu_0$ is trivial.
\begin{proof}
$\nabla_{\partial_t \mu_t}$ is a derivation of $[-,-]$ and $[-,-](\mu_t)=\mu_t$. Hence, the parallel transport $P_{\mu}^{t,0}(\nabla): \g\rightarrow \g$ along $\mu_t$ gives a Lie algebra isomorphism between $(\g,\mu_0)$ and $(\g,\mu_t)$. 
\end{proof}
\end{coro}
Indeed, the parallel transport can be thought of as a smooth function $P_\mu(\nabla): I\rightarrow GL(\g)$ and the Lie algebra isomorphism $(\g, \mu_0)\cong (\g,\mu_t)$ is given by the change of coordinates induced by the parallel transport, i.e.
\[
	\mu_0(x,y)= \left(P^{t,0}_\mu(\nabla)\right)^{-1} \mu_t( P_\mu^{t,0}(\nabla)(x), P_\mu^{t,0}(\nabla)(y)). 
\]
In other words 
\begin{equation}\label{abc}
\mu_0= \mu_t\cdot P_{\mu}^{t,0}(\nabla).
\end{equation} 
\begin{thm}\label{rigidityPT}
If $H^1(\mu_0)=0$ then $\mu_0$ is rigid. 
\begin{proof}
The space of Lie algebra structures $Lie(\g)$ is a quadratic affine algebraic variety. Because of \cite[Theorem 9.3.6]{BCR1998} we have that it is locally path connected. 
 
This fact, together with \ref{abc}, implies the result.
\end{proof}
\end{thm}

\section{An introduction to $L_\infty$-algebras}
In this section we follow chapter $2$ of \cite{baarsma2019deformations}. 

\subsection{The graded symmetric algebra}
Let $V=\bigoplus_{i\in \mathbb{Z}} V_i$ be a graded vector space. An element $v\in V$ is called \textit{homogeneous} if $v\in V_k$ for some $k\in\Z$. Its \textit{degree} is denoted by $\grau{v}:=k$. The tensor algebra of $V$ can be decomposed by degree and rank via
\[
		\bigotimes V= \bigoplus_{k\in \mathbb{N}} \ \bigoplus_{i\in\mathbb{Z}} \ \bigoplus_{n_1+\ldots +n_k=i}  \otimes_{j=1}^k V_{n_j},
\]
with $k$ the rank and $i$ the degree. Accordingly, $(\bigotimes V, \otimes)$ becomes a graded algebra with respect to the degree. Let $\mathcal{I}$ be the graded ideal generated by 
\[
	\mathcal{I}:=\langle u\otimes v- (-1)^{\grau{u}\grau{v}} v\otimes u \mid u,v\in V \text{ homogeneous} \rangle.
\]	
The quotient $S(V):= \bigotimes V/ \mathcal{I}$  has the induced algebra structure $\odot$ from $\otimes$ and is called the \textit{graded symmetric algebra of} $V$. It is useful to introduce a sign rule that takes care of these symmetries.
\begin{dfn}
Let $\sigma\in S_k$ be a permutation and $v_1,\ldots, v_k\in V$ homogeneous elements. The \textit{Koszul sign rule} is given by 
\[
	\epsilon_\sigma(v_1, \ldots, v_k):= \prod_{\substack{i<j \\ \sigma(i)> \sigma(j)}} (-1)^{\grau{v_i}\grau{v_j}}.
\]
\end{dfn}
We would write $\epsilon_\sigma$ when there is no risk of confusion. The sign rule is such that 
\[
	v_{\sigma(1)}\odot \cdots \odot v_{\sigma(k)}= \epsilon_\sigma \ v_1\odot \cdots \odot v_k.
\]
\subsection{Plurilinear maps}
Let $V$ and $W$ be two graded vector spaces. 
\begin{dfn}
A linear map $f: \bigotimes^k V\rightarrow W$ is called \textit{graded symmetric} if, for any two consecutive indices $i$ and $i+1$, we have that
\[
	f(v_1, \ldots,v_i, v_{i+1}, \ldots, v_k)= (-1)^{\grau{v_i}\grau{v_{i+1}}} f(v_1, \ldots,v_{i+1}, v_{i}, \ldots, v_k).
\]
$f$ is called \textit{homogeneous of degree} $d\in \Z$ if, for every $k$-tuple of homogeneous elements $v_1, \ldots, v_k\in V$, we have that
\[
	f(v_1,\ldots, v_k)\in W_{\grau{v_1}+ \cdots + \grau{v_k}+d}.
\]
\end{dfn}
Let $Lin\left( S^k(V), W\right)_d$ be the vector space of graded symmetric $k$-multilinear maps of degree $d$. 
\begin{dfn}
The space of \textit{plurilinear maps of degree} $d$ is the vector space given by
\[
	Lin\left( S(V), W\right)_d:= \prod_{k\in \N}  Lin\left( S^k(V), W\right)_d.
\]
\end{dfn}
An element $\ell\in Lin\left( S(V), W\right)_d$ is given by a sequence $(\ell_k)_{k\in\N}$ such that $\ell_k\in Lin\left( S^k(V), W\right)_d$. In particular $\ell_0: \R\rightarrow V_d$ is identified with its image $\ell_0(1)\in V_d$.
\subsection{$L_\infty$-algebras}
Recall that a permutation $\sigma\in S_{p+q}$ is a $(p,q)$-\textit{unshuffle} if $\sigma(1)<\cdots< \sigma(p)$ and $\sigma(p+1)<\cdots< \sigma(p+q)$. We denote the \textit{space of $(p,q)$-unshuffles} by $S_{p,q}$.
\begin{dfn}
The map $Jac_n: Lin\left( S(V), W\right)_d\rightarrow Lin\left( S^n(V), W\right)_{2d}$, given by 
\begin{multline}\label{jacobi}
	Jac_n(\ell)(v_1,\ldots, v_n)=  \\ 
	\sum_{\substack{i+j=n+1 \\ i\leq j}}\ \sum_{\sigma\in S_{i, j-1} } \epsilon_\sigma \ell_j\left( \ell_i\left(v_{\sigma(1)}, \ldots, v_{\sigma(i)} \right), v_{\sigma(i+1)}, \ldots, v_{\sigma(n)} \right),
\end{multline}
is called the $n$-\textit{Jacobiator}.
\end{dfn}
\begin{exm}
Let $\ell\in  Lin\left( S(V), W\right)_1$ be such that $\ell_k=0$ for $k\neq 1,2$. Then:
\begin{enumerate}
\item  $Jac_1(\ell)=0$ if and only if $(V,\ell_1)$ is a cochain complex. 
\item  $Jac_3(\ell)=0$ if and only if $(V,\ell_2)$ is a graded Lie algebra. 
\item  $Jac_1(\ell)=Jac_2(\ell)=Jac_3(\ell)=0$ if and only if $(V,\ell_1,\ell_2)$ is a differential graded Lie algebra. 
\end{enumerate}
\end{exm}
\begin{dfn}
A \textit{curved} $L_\infty$-\textit{algebra structure on} $V$ is a plurilinear map of degree one $\ell\in Lin\left( S(V), W\right)_1$ such that $Jac_n(\ell)=0$ for all $n\in \N$. Its \textit{curvature} is the element $\ell_0\in V_1$. An $L_\infty$-\textit{algebra} is a curved $L_\infty$-algebra with zero curvature.
\end{dfn}
\begin{rmk}
$L_\infty$-algebras were introduced in \cite{lada1992introduction}. The definition given in this paper is usually encountered as $L_\infty[1]$-algebra structures. Using the shifted graded vector space $(V[1])_d:= V_{d+1}$ and the décalage isomorphism, one can show that the definitions in \textit{loc. cit.} and in this paper are equivalent.  A textbook account of the use of $L_\infty$-algebras in deformation theory can be found in \cite{Manetti2022}.
\end{rmk}

\subsection{Twisted $L_\infty$-algebras and the Maurer-Cartan equation}
Let $(V,\ell)$ be an $L_\infty$-algebra and $u\in V_0$. 
\begin{dfn}
The \textit{twisting} of $\ell$ by $u$ is the graded map $\ell^u\in Lin( S(V), V)_1$ given by 
\begin{equation*}
	\ell^u_p(v_1,\ldots, v_p):= \sum_{k\in\N} \dfrac{1}{k!}\ell_{p+k}(\odot^k u\odot v_1\odot\cdots\odot v_p). 
\end{equation*}
\end{dfn}
Let $C_\ell=\lbrace u\in V_0\mid \ell^u \text{ converges}\rbrace$. 
\begin{dfn}
The \textit{domain of convergence} of a curved $L_\infty$-algebra $(V,\ell)$ is the interior of $C_\ell$. We denote it by $D_\ell:= Int(C_\ell)$. 
\end{dfn}
The following result is given by Propositions $2.4.3$ and $2.4.4$ of \cite{baarsma2019deformations}.
\begin{prop}\label{derivativeproperty}
For every curved $L_\infty$-algebra $(V,\ell)$, the map 
\[
	\mathfrak{l}: D_\ell\rightarrow Lin( S(V),V)_1 , \ \ \ u\mapsto \ell^u,
\]
is real analytic and its derivatives are given by 
\begin{equation*}
	D^k\mathfrak{l} (u)(\odot^k \dot{u})(v_1, \ldots, v_p)= \ell^u_{k+p}(\odot^k \dot{u}\odot v_1\odot\cdot\odot v_p).
\end{equation*}
Moreover, $(V,\mathfrak{l}(u))$ is a curved $L_\infty$-algebra.
\end{prop}
\begin{coro}\label{derivative}
For every smooth path $u_t\subset D_\ell$ we have that 
\begin{multline*}
	\partial_t\ell_i^{u_t} (\partial_t^{r_1} u_t \odot\cdots\odot \partial_t^{r_i} u_t) =\ell_{i+1}^{u_t}( \partial_t u_t\odot \partial_t^{r_1}u_t\odot\cdots\odot \partial_t^{r_i} u_t)\\
	+\sum_{j=1}^i \ell_i^{u_t}(\partial_t^{r_1} u_t\odot\cdots\odot \partial_t^{r_j+1} u_t\odot\cdots\odot \partial_t^{r_i} u_t).
\end{multline*}
\end{coro}
\begin{proof} We can write $\ell_i^{u_t} (\partial_t^{r_1} u_t \odot\cdots\odot \partial_t^{r_i} u_t)$ as $\mathfrak{l}(u_t) (\partial_t^{r_1} u_t \odot\cdots\odot \partial_t^{r_i} u_t)$. Using the chain rule the result follows.
\end{proof}

The \textit{Maurer-Cartan equation} measures how far the twisted \textit{curved} $L_\infty$-algebra $\ell^u$ is from being an $L_\infty$-algebra. Explicitly, $MC: D_\ell\subset V_0\rightarrow V_1$ is defined by $MC(u):= \ell_0^u$.  
\begin{dfn}
$u\in D_\ell$ is a \textit{Maurer-Cartan element} if $u\in\ MC^{-1}(0)$. Explicitly  
\begin{equation*}
	MC(u)= \ell^u_0= \sum_{k\in\N} \dfrac{1}{k!}\ell_{k}(\odot^k u)=0. 
\end{equation*}
We denote by $MC(V,\ell)$ the set of Maurer-Cartan elements of $(V, \ell)$. 
\end{dfn}

\section{Deformations of Maurer-Cartan elements and obstructions} 

\subsection{Deformations and obstructions}

Let $(V,\ell)$ be an $L_\infty$-algebra and let $u_0\in MC(V,\ell)$.

\begin{dfn}
 A \textit{deformation of} $u_0$ is a smooth path $u_t\subset D_\ell$ that starts at $u_0$ and satisfies $MC(u_t)=0$. 
\end{dfn}
We want to find obstructions to the existence of deformations $u_t$ of $u_0$. Let $u_t\subset D_\ell$ be any smooth path starting at $u_0$.
Natural necessary conditions for $MC(u_t)=0$ to hold are the differential consequences of this equation, namely $\partial_{t=0}^k MC(u_t)=0$. Denote $\partial_{t=0}^k u_t= u_0^k$ and use the notation $\vec{r}_{i}=k$ of Appendix \ref{appendix-partition} for partitions of a number.

\begin{prop}\label{dMC1}
	For every $k\in\N$ we have the equation  
		\[
			\partial_{t=0}^{k} MC(u_t) =\ell_1^{u_0}(u_0^{k})+\sum_{i=2}^{k} \sum_{\vec{r}_{i}=k} \binom{k}{\vec{r}_{i}} \dfrac{1}{i!} \ \ell_{i}^{u_0}(u_0^{r_1} \odot \cdots\odot u_0^{r_i}).
		\]
	\begin{proof}
	Since $MC(u_t)= \ell^{u_t}_0$ then $\partial_t (MC(u_t))= \ell^{u_t}_1(\partial_t u_t)$. We now make use of Proposition \ref{PropAuxB}, taking $v_t=\partial_t u_t$, to conclude that 
		\begin{multline*}
			\partial_{t=0}^{k} MC(u_t) =\ell_1^{u_0}(u_0^{k})+ \\
			\sum_{i=1}^{k-1} \sum_{j=0}^{k-1-i}\sum_{\vec{r}_{i}=k-1-j} \binom{k-1}{\vec{r}_{i}} \dfrac{1}{i!j!}\ \ell^{u_0}_{i+1}(u_0^{r_1}\odot\cdots\odot u_0^{r_{i}}\odot u_0^{j+1}).
		\end{multline*}	
	Reorganizing the terms we get 
		\[
			\partial_{t=0}^{k} MC(u_t)=\ell_1^{u_0}(u_0^{k})+\sum_{i=2}^k \sum_{\vec{r}_i=k} \binom{k-1}{\vec{r}_i}\dfrac{r_i}{(i-1)!} \ \ell_i^{u_0}(u_0^{r_1}\odot \cdots \odot u_0^{r_i}).
		\]
	Moreover, for every pair of summands $r_p$ and $r_q$ of $r_1+\cdots +r_i=k$, we have that 
	\[
		\sum_{\vec{r}_i=k} \binom{k-1}{\vec{r}_i}\dfrac{r_p}{(i-1)!} \ \ell_i^{u_0}(u_0^{r_1}\odot \cdots \odot u_0^{r_i})=\sum_{\vec{r}_i=k} \binom{k-1}{\vec{r}_i}\dfrac{r_q}{(i-1)!} \ \ell_i^{u_0}(u_0^{r_1}\odot \cdots \odot u_0^{r_i}).
	\]
	Thus
	\[ 
		\begin{split}
		 & i \cdot \sum_{\vec{r}_i=k} \binom{k-1}{\vec{r}_i}\dfrac{r_i}{(i-1)!} \ \ell_i^{u_0}(u_0^{r_1}\odot \cdots \odot u_0^{r_i}) \\
		& =\sum_{\vec{r}_i=k} \binom{k-1}{\vec{r}_i}\dfrac{r_1+\cdots +r_i}{(i-1)!} \ \ell_i^{u_0}(u_0^{r_1}\odot \cdots \odot u_0^{r_i}) \\
		& =\sum_{\vec{r}_i=k} \binom{k}{\vec{r}_i}\dfrac{1}{(i-1)!} \ \ell_i^{u_0}(u_0^{r_1}\odot \cdots \odot u_0^{r_i}), \\
		\end{split}
	\]
and the result follows. 
\end{proof}
\end{prop}

We will focus on building paths $u_t$ that satisfy this sequence of necessary conditions.

\begin{dfn} A $k$-\textit{deformation of} $u_0$ is a  path $u_t\subset D_\ell$ that starts at $u_0$ and satisfies $\partial_{t=0}^j MC(u_t)=0$ for all $0\leq j\leq k$. It is called $(k+1)$-\textit{extendable} if there exists a $(k+1)$-deformation $v_t$ of $u_0$ such that  $v_0^j=u_0^j$ for all $0\leq j\leq k$.
\end{dfn}

The problem of $(k+1)$-extending a $k$-deformation $u_t$, i.e. finding an appropriate $v_0^{k+1}$, amounts to solving the equation \begin{equation}\label{indeq}
	0 = \ell_1^{u_0} (v^{k+1}_0)+ Obs^{k}(u_{t}),
\end{equation}

where the last term of the equation is the following obstruction.

\begin{dfn}
Let $k\in\N$. The \textit{obstruction to $(k+1)$-extendability} of a $k$-deformation $u_t$ is the element $Obs^{k}(u_{t})\in V_1$ given by 
\begin{equation}\label{obstructions}
	Obs^{k}(u_t)= \sum_{i=2}^{k+1} \sum_{\vec{r}_{i}=k+1} \binom{k+1}{\vec{r}_{i}} \dfrac{1}{i!} \ \ell_{i}^{u_0}(u_0^{r_1} \odot \cdots\odot u_0^{r_i}).
\end{equation}
\end{dfn}

Recall that, by the first Jacobi identity \ref{jacobi},  $(V, \ell_1^{u_0})$ is a cochain complex. 
We will now see that the obstructions are cocycles.
\begin{prop}\label{aux}
If $u_t$ is a $k$-deformation of $u_0$ then
$\ell_1^{u_0}(Obs^{k}(u_{t}))=0$. In particular it defines a class in $H^1(V, \ell^{u_0})$. 
\end{prop} 

In order to prove this proposition we make use of the following lemma.

\begin{lemma}\label{dMC2}
	For every $k\in\N$ we have the following equation 
		\begin{multline*}
			\ell_1^{u_t}(\partial_t^k MC(u_t)) \\
			+\sum_{i=1}^{k} \sum_{j=0}^{k-i}\sum_{\vec{r}_{i}=k-j} \binom{k}{\vec{r}_{i}} \dfrac{1}{i!j!}\ \ell^{u_t}_{i+1}(\partial_t^{r_1} u_t \odot\cdots\odot \partial_t^{r_i} u_t \odot \partial_t^j MC(u_t))=0.
		\end{multline*}
\begin{proof}
	 By Proposition \ref{PropAuxB}, letting $v_t=MC(u_t)$, the left-hand side equals 
	 \begin{multline*}
			\partial_{t}^{k} \ell_1^{u_t} MC(u_t) =\ell_1^{u_t}(\partial^k_t MC(u_t))+ \\
			\sum_{i=1}^{k} \sum_{j=0}^{k-i}\sum_{\vec{r}_{i}=k-j} \binom{k}{\vec{r}_{i}} \dfrac{1}{i!j!}\ \ell^{u_t}_{i+1}(\partial_t^{r_1} u_t \odot\cdots\odot \partial_t^{r_i} u_t \odot \partial_t^j MC(u_t)).
		\end{multline*}	
	By the 0-th Jacobi identity \ref{jacobi}, $\ell_1^{u_t} MC(u_t)=\ell_1^{u_t}( \ell_0^{u_t})=0$.
\end{proof}
\end{lemma}

\begin{proof}[Proof of Proposition \ref{aux}]
By hypothesis $\partial_{t=0}^k MC(u_t)=0$. Using Lemma \ref{dMC2}, we get that
\[
	\ell_1^{u_0}(\partial_{t=0}^{k+1} MC(u_t))=0.
\]
By Proposition \ref{dMC1}, 

\begin{equation}\label{MC_OBS}
\partial_{t=0}^{k+1} MC(u_t)=\ell_1^{u_0}(u_0^{k+1})+Obs^{k}(u_t),
\end{equation} and the result follows.
\end{proof}

We can now recast Equation \ref{indeq} as an equation in cohomology.

\begin{thm}
Let $u_0\in MC(V,\ell)$ be a Maurer-Cartan element and $u_t\subset D_\ell$ a deformation of $u_0$. Then 
\[
	Obs^k(u_t)=0 \ \ \text{in} \ \ H^1(V,\ell^{u_0}),
\]
for every $k\in\N$. 
\begin{proof}
If $u_t$ is a deformation of $u_0$, it is in particular a $k$-deformation and $k+1$-extendable, for all $k$. The result follows from Proposition \ref{aux} and Equation \ref{indeq}. 
\end{proof}
\end{thm}

\subsection{Formal Maurer-Cartan elements}
Now we consider the formal power series $V [\! [t] \!] := \oplus_{i\in\Z} V_i [\! [t] \!]$. We can extend the $L_\infty$-algebra framework to the formal setting by letting all the structure maps be $t$-linear. 

\begin{dfn}A \textit{formal Maurer-Cartan element} of $(V, \ell)$ is a Maurer-Cartan element of $(V [\! [t] \!], \ell)$.
\end{dfn}

Let $u[\! [t] \!]\in V_0[\! [t] \!]$ be $u[\! [t] \!]= \sum_{k\geq 0} \dfrac{u_k}{k!} t^k$. By the Equations \ref{MC_OBS} we know that $u[\! [t] \!]$ is a formal Maurer-Cartan element if and only if 
\begin{equation}\label{melu}
	 \ell_1^{u_0} (u_{k+1})+ Obs^{k}(u[\! [t] \!])=0, \ \ \text{for all} \ \ k.
\end{equation}
By definition, $Obs^{k}(u[\! [t] \!])$ depends only on $u_0, \ldots, u_k$. Therefore we can construct a formal Maurer-Cartan element by solving a recursive sequence of cohomological equations. 
\begin{thm}\label{formalint}
Let $v_0$ be a Maurer-Cartan element of $(V,\ell)$. If $H^1(V,\ell^{v_0})=0$ then, for every $v_1\in \ker \ell^{v_0}_1$, there exists a formal Maurer-Cartan element $u[\! [t] \!]$ such that $u_0=v_0$ and $u_1=v_1$. 
\begin{proof}
Let $u_0=v_0$ and $u_1=v_1\in \ker \ell^{v_0}_1$ and suppose we have $u_0,\ldots, u_{k}$ such that \ref{melu} holds for $1\leq j\leq k-1$. By Proposition \ref{aux}, letting 
\[
	u^k[\! [t] \!]=\sum_{j=0}^k \dfrac{u_j}{j!} t^j,
\]
 we have that $Obs^{k}(u^k[\! [t] \!])\in H^1(V,\ell^{v_0})$. Consequently, there exists $u_{k+1}$ such that 
\[
	 \ell_1^{v_0} (u_{k+1})+ Obs^{k}(u^k[\! [t] \!])=0.
\]
Therefore, we can extend the $k$-deformation $u^k[\! [t] \!]$ to the $(k+1)$-deformation 
\[
	u^{k+1}[\! [t] \!]=\sum_{j=0}^{k+1} \dfrac{u_j}{j!} t^j.
\]
\end{proof}
\end{thm}

\begin{rmk} The same result, when $(V, \ell)$ is a differential graded Lie algebra, is classic and can be found for example in \cite[Theorem 3.25]{DMZ2007} (in the context of deformations of associative algebras). Although the result for general $(V, \ell)$ could then be obtained via rectification (see \cite[Appendix B6]{Quillen1969}, or \cite[Section 2.2]{Hinich2001}) that process leads to infinite dimensional differential graded Lie algebras. 

We have chosen to offer a construction of the obstructions directly in terms of $(V, \ell)$. When $V$ is finite dimensional, the obstructions turn out to be convenient for the study of convergence of formal Maurer-Cartan elements, explicitly realizing them as smooth families of Maurer-Cartan elements. We will focus on that task in the next Section.
\end{rmk}

\section{Integrability}
We come back to the set up of Section \ref{Sec-rigidity}: 
\begin{equation}\label{problem}
\begin{tikzcd}
					 E\curvearrowleft G \arrow[d]  \\
					\arrow[ u, bend left, dashed, "\sigma"]  M \curvearrowleft G
				\end{tikzcd}
\end{equation}

\begin{dfn} We say that the problem \ref{problem} is \textit{integrable} at $x_0$ if there exists an open subset $x_0\in U\subset M$ and a (immersed) submanifold $S_{x_0}\subset M$ such that $S_{x_0}\cap U=\sigma^{-1}(0)\cap U$.
\end{dfn}
\begin{rmk}
 If we have integrability at $x_0$ then $G(x_0)\cap U\subset S_{x_0}\cap U$. Furthermore, by Proposition \ref{Tsigma}
\[
	\dim G(x_0)\leq \dim S_{x_0}\leq \dim \ker d_{x_0}^v\sigma.
\] 
\end{rmk}
\begin{dfn} The problem \ref{problem}
 is called \textit{maximally integrable} at $x_0$ if it is integrable of maximal dimension, i.e. $\dim S_{x_0}= \dim \ker d_{x_0}^v\sigma$. 
\end{dfn}

\subsection{Stable sections}

In order to motivate the definition of stable sections, and of equivariant deformation problems, we recall one of the main elements of the proof of Proposition \cite[Proposition 4.4 (1)]{crainic2014survey}. 

Integrability is a local question, therefore we consider $E=M\times V$, with $V$ some vector space and $\sigma:M\rightarrow V$. We want $\sigma^{-1}(0)$ to have a manifold structure. Of course, the first thing that comes to mind is to use the regular value theorem. However, being a submersion is very restrictive and we want to relax this hypothesis. An alternative is to use transversality. In fact, picking a complement $V= Im \ d_{x_0} \sigma \oplus B$ gives us a manifold structure for $\sigma^{-1}(B)$. How far is $\sigma^{-1}(B)$ from being $\sigma^{-1}(0)$? The following extra structure controls this question. 
\begin{dfn}

Let $\pi_F: F\rightarrow M$ be a vector bundle and let  $\Phi: E\rightarrow F$ be a vector bundle map over the identity.
The section $\sigma: M\rightarrow E$ is called \textit{stabilized by} $\Phi$ if $\Phi\circ \sigma=0$. 
\end{dfn}
Indeed, since $\Phi\circ\sigma=0$, we know that $\sigma(M)\subset \ker \Phi$. Accordingly, the following cochain complex controls how far is $\sigma^{-1}(B)$ from being equal to $\sigma^{-1}(0)$:
\begin{equation}\label{stableseq}
		\begin{tikzcd}
			  T_{x_0} M \arrow[r," d_{x_0}^v\sigma"] & E_{x_0} \arrow[r, " \Phi_{x_0}"] & F_{x_0}.
		\end{tikzcd}
\end{equation}
\begin{dfn}
The solution $x_0$ is called \textit{infinitesimally stable} if the sequence \ref{stableseq} is exact at $E_{x_0}$. 
\end{dfn}

In \cite{crainic2014survey}, the authors prove the following.
\begin{prop}
Let $x_0\in \sigma^{-1}(0)$. If $x_0$ is infinitesimally stable, the problem is maximally integrable at $x_0$. 
\end{prop}

We will use homotopy operators for a suitable $L_\infty$-algebra to give a new proof of this result in Corollary \ref{MaxIntNew}.

\begin{rmk}
A section can be stabilized in different ways. For example, one can always take $F\rightarrow M$ an arbitrary vector bundle and $\Phi=0$. In this case being infinitesimally stable just means that $0$ is a regular value of $\sigma$. 
\end{rmk}

\subsection{Equivariant deformation problems and $L_\infty$-algebras}
\begin{dfn}
		An (analytic) equivariant deformation problem is given by 
			\[
				\begin{tikzcd}
					 E\curvearrowleft G \arrow[d]  \arrow[r,"\Phi"] & F\arrow[dl]\\
					\arrow[ u, bend left, dashed, "\sigma"]  M \curvearrowleft G
				\end{tikzcd}
			\]
		where:
		\begin{enumerate}[i.]
			\item $E \rightarrow M$ is an (analytic) vector bundle whose zero section $0_M$ is $G$-invariant.  
			\item $\sigma\in\Gamma_G(E)$ is an (analytic) $G$-equivariant section. 
			\item $\pi_F:F\rightarrow M$ is an (analytic) vector bundle.
			\item $\Phi: E\rightarrow F$ is an (analytic) vector bundle map such that $\Phi\circ \sigma=0$. 
		\end{enumerate}
We denote an (analytic) deformation problem by $(M\overset{\sigma}{\rightarrow} E \overset{\Phi}{\rightarrow} F)\curvearrowleft G $.  
\end{dfn} 
	\begin{exm}{\textbf{ Lie algebra structures:}}
		Let $\g$ be a finite dimensional vector space. Recall that  $C^k(\g)= Hom( \wedge^k \g,\g)$ and that $GL(\g)$ acts on $C^k(\g)$ by 
		\[
			(\eta\cdot A)(x_1,\ldots, x_k):= A^{-1}\eta(Ax_1, \ldots, Ax_k). 
		\]
		Then 
		\[
				\begin{tikzcd}
					 C^2(\g)\times C^3(\g) \curvearrowleft GL(\g) \arrow[d]  \arrow[r,"\Phi"] & C^2(\g)\times C^4(\g) \arrow[dl]\\
					\arrow[ u, bend left, dashed, " Jac"]  C^2(\g) \curvearrowleft GL(\g) 
				\end{tikzcd}
		\]
		is an analytic equivariant deformation problem with $Jac$ the Jacobiator and $\Phi(\mu, \eta):= \delta_\mu \eta$, where $\delta_\mu$ is the Chevalley-Eilenberg operator \begin{multline*}\delta_\mu\eta(v_1,\ldots,v_4)=\sum_{i=1}^4(-1)^{i+1}\mu(v_i,\eta(v_1,\ldots,\hat{v_i},\ldots,v_4)\\ + \sum_{1\leq i<j\leq 4} (-1)^{i+j}\eta(\mu(v_i,v_j),v_1,\ldots,\hat{v_i}\ldots,\hat{v_j},\ldots,v_4).\end{multline*} The space $Jac^{-1}(0)= Lie(\g)$ is the space of Lie algebra structures on $\g$. 
	\end{exm}
Equivariant deformation problems and curved $L_\infty$-algebras are closely related. To see this, take a chart $\varphi:U\subset M\rightarrow \R^n$, with $x_0\mapsto 0$, such that the vector bundles $E\vert_{U}\cong U\times A$ and $F\vert_U\cong U\times B$ are trivialized. Consider the graded vector space $V= \g[-1]\oplus \R^n[0]\oplus A[1]\oplus B[2]$, where $[i]$ indicates the grading. In the PhD thesis \cite{baarsma2019deformations} the author proved (a more general version) of the following:
\begin{thm}[\cite{baarsma2019deformations}, Theorem 5.25]\label{arjen}
Let $(M\overset{\sigma}{\rightarrow} E \overset{\Phi}{\rightarrow} F)\curvearrowleft G $ be an analytic equivariant deformation problem and let $\varphi$, $U$ and $V$ be as above. Then $V$ admits a curved $L_\infty$-algebra structure $\ell\in Lin(\bigodot V, V)_1$ such that there exists an open subset $U^\prime \subset U$ for which
\[
	\varphi: U^\prime \cap \sigma^{-1}(0)\rightarrow \varphi(U^\prime)\cap MC(V, \ell)
\] 
is a bijection.
\end{thm}
\begin{rmk}
The curved $L_\infty$-algebra $(V,\ell)$ is determined by the Taylor series at $x_0$ of: the section $\sigma$, the actions of $G$ on $M$ and on $E$, the Lie bracket on $\g$, and the vector bundle map $\Phi$ \cite[Theorem 5.2.4]{baarsma2019deformations}. Moreover, if we choose different trivializations we get isomorphic curved $L_\infty$-algebras. 
\end{rmk}
In particular, the Maurer-Cartan equation $MC: \R^n\rightarrow A$ is given by
\[
	MC(v)= \sum_{k\geq 0} \dfrac{1}{k!} d^k_0\tsigma (v, \overset{k}{\cdots}, v),
\]
with $\tsigma: \R^n\rightarrow A$ the local description of the section.  Accordingly \[\sigma(x_0)=0 \ \ \Leftrightarrow\ \ MC(0)=0 \ \ \Leftrightarrow\ \ (V,\ell)\ \text{is an } L_\infty\text{-algebra.}\]

\subsection{Homotopy operators and integrability}
Let $(M\overset{\sigma}{\rightarrow}E\overset{\Phi}{\rightarrow} F)\curvearrowleft G$ be an analytic deformation problem, with $\sigma(x_0)=0$, and let $ (V, \ell)$ be the $L_\infty$-structure associated to it by Theorem \ref{arjen}. The chart $\varphi: U\subset M\rightarrow \R^n$ gives an equivalence between integrability of the problem at $x_0$ and smoothness of an open neighbourhood $0\in W\cap MC(V,\ell)$ of the Maurer-Cartan elements. We will use the algebraic structure on $V$ to construct an explicit smooth structure on $W\cap MC(V,\ell)$. Indeed, Theorem \ref{formalint} tells us how the cohomology $H^1(V,\ell)$ controls formal Maurer-Cartan elements $u_t$. Letting 
\[
	u_t=\sum_k \dfrac{u_k}{k!} t^k,
\] 
we showed that the coefficients of $u_t$ can be built recursively by solving the infinite sequence of cohomological equations
\[
	\ell_1(u_{k+1})+Obs^k(u[\![t ]\!])=0. 
\]
Now, by Proposition \ref{hom}, the vanishing of $H^1(V,\ell)=0$ gives us homotopy operators $h_2: B\rightarrow A$ and $h_1: A\rightarrow \R^n$. They provide explicit solutions for the coefficients of $u_t$. 
\begin{prop}\label{solution}
Let $u_0=0$ and $u_1\in \ker \ell_1$. If $H^1(V,\ell)=0$ then $u_t=\sum_k \dfrac{u_k}{k!} t^k\in \R^n[ \![ t]\!]$, given by 
\[
	u_{k+1}:=- h_1 ( Obs^k( u_t)),
\]
is a formal Maurer-Cartan element. 
\end{prop}

\begin{rmk} When $ (V, \ell)$ is a differential graded Lie algebra, a similar recurrence has recently been used to construct formal Maurer-Cartan elements, in the context of perturbative quantum mechanics \cite{LS2024}.
\end{rmk}

Our objective now is to use these explicit solutions to prove that $u_t$ converges for small values of $t$ and $u_1$. To see this, take a norm  $\norm{\cdot}$ on $\R^n$. By the equation \ref{obstructions} we get that 
\begin{equation}\label{aux2}
	\dfrac{\norm{u_{k+1}}}{(k+1)!}\leq \norm{h_1}\sum_{i=2}^{k+1}  \dfrac{\norm{\ell_{i}}}{i!} \sum_{\vec{r}_{i+1}=k+1 } \dfrac{\norm{u_{r_1}}}{(r_1)!} \cdots \dfrac{\norm{u_{r_{i+1}}}}{(r_{i+1})!}.
\end{equation}
\begin{lemma}\label{boundcoeff}
Suppose that there exists $\alpha>0$ satisfying
\[
	\sum_{i=1}^\infty \dfrac{\norm{\ell_{i}}}{i!}\leq \alpha.
\]
Then, for every $k$ we have that
\[
	\dfrac{\norm{u_{k}}}{k!}\leq \norm{u_1}^{k} (\norm{h_1}\alpha)^{k-1} C_{k} , 
\]
with 
\[
	C_k=\sum_{i=2}^{k}\sum_{\vec{r}_i=k} C_{r_1}\cdots C_{r_i}.
\]
\begin{proof}
If $\norm{h_1}\alpha\leq 1$ then the result follows by Inequality \ref{aux2}. For $\norm{h_1}\alpha\geq 1$, the proof is by strong induction. By Proposition \ref{solution} we know that 
\[
	u_2= -h_1\circ \ell_2(u_1,u_1),
\]
and letting $C_1=C_2=1$ the base cases holds. Assume now that the result holds for every $p\leq k-1$. By Inequality \ref{aux2}, we conclude that
\[
\dfrac{\norm{u_{k}}}{k!}\leq (\norm{u_1}\norm{h_1} \alpha)^{k} \sum_{i=2}^k   \left(\dfrac{1}{\norm{h_1}\norm{\alpha}}\right)^i \sum_{\vec{r}_i=k} C_{r_1}\cdots C_{r_i}.
\]
But $\left(\dfrac{1}{\norm{h_1}\norm{\alpha}}\right)^i \leq \dfrac{1}{\norm{h_1}\norm{\alpha}}$ for all $i$ and the result follows. 
\end{proof}
\end{lemma}

\subsection{Integrability of $N$-strict $L_\infty$-algebras}

\begin{dfn}We say that an $L_\infty$-algebra is \textit{$N$-strict} if $\ell_k=0$ for all $k\geq N$.
\end{dfn}

For example, any nilpotent $L_\infty$-algebra is $N$-strict \cite[Section 10.5]{Manetti2022}. Given an $N$-strict $L_\infty$-algebra, and any norm, let \[\alpha_\ell =\sum_{i=1}^N \frac{\norm{\ell_i}}{i!}.\]

\begin{thm}\label{intpath}
Let $(V,\ell)$ be an $N$-strict $L_\infty$-algebra such that $H^1(V,\ell)=0$. Then, for every $u_1\in \ker\ell_1$ with 
\[
	\norm{u_1}< \dfrac{1}{12 \norm{h_1}\alpha_\ell},
\]
there exists a deformation $u_t: [0,2)\rightarrow MC(V,\ell)\subset \R^n$ of $0$ by Maurer-Cartan elements, such that $\partial_{t=0}u_t=u_1$. 
\begin{proof}
Let $u_t$ be given as in Proposition \ref{solution}.  
By Lemma \ref{boundcoeff} we know that 
\[
	\dfrac{\norm{u_{k}}}{k!}\leq \norm{u_1}^{k} (\norm{h_1}\alpha_\ell)^{k-1} C_{k}.
\]
Moreover, the numbers $C_k$ are the super-Catalan numbers (see Proposition \ref{SCN}) and by the asymptotic growth \ref{asymp} we know that, for $k>>1$
\[
	\dfrac{\norm{u_{k}}}{k!}\leq \dfrac{1}{\norm{h_1}\alpha_\ell} (6\norm{u_1} \norm{h_1}\alpha_\ell)^k .
\]
Therefore, there exists a constant $M>0$ such that 
\[
	\norm{u_t}\leq M+ \dfrac{1}{\norm{h_1}\alpha_\ell} \sum_{k >>1} \left( 6 \norm{u_1} \norm{h_1}\alpha_\ell t \right)^{k+1},
\]
and $u_t$ converges. Finally, since $u_t$ is analytic by construction and the twisting of an $L_\infty$-algebra is analytic by Proposition \ref{derivativeproperty}, then $MC(u_t)$ is an analytic map. Thus $MC(u_t)=0$ if and only if $\partial^k_{t=0} MC(u_t)=0$ for all $k$. We conclude that $u_t: [0,2) \rightarrow MC(V,\ell)$ takes values in the Maurer-Cartan elements. 
\end{proof}
\end{thm}
Let 
\[
	B_{h_1,\ell}:= \left\{ v\in \ker \ell_1\mid \norm{v}< \dfrac{1}{12 \norm{h_1}\alpha_\ell}\right\},
\]
and define 
\[
	\begin{split}
		\psi: B_{h_1,\ell}&\rightarrow C^\infty([0,2),\R^n) \\
		  v&\mapsto v_t
		\end{split}
\]
the map that associates to each cocycle $v$ the path $v_t$ given by Theorem \ref{intpath}. 
\begin{lemma}
For every $s\in [0,1]$ and $v\in B_{h_1,\ell}$ we have that $\psi(sv)(t)=\psi(v)(st)$. 
\begin{proof}
Let $v_t=\psi_h(v)(t)$ and $(sv)_t=\psi_h(sv)(t)$. Denote their coefficients by $v_k$ and $(sv)_k$ respectively. We claim that $(sv)_{k+1}= s^{k+1} v_k$. Indeed, by Proposition \ref{solution} and definition \ref{obstructions} it is easy to see that this holds. Consequently 
\[
	\psi_h(v)(st)= \sum_{k\geq 0} v_k (st)^k = \sum_{k\geq 0} (sv)_k t^k = \psi(sv)(t).\qedhere 
\]
\end{proof}
\end{lemma}

\begin{thm}\label{intdgla}
Let $\Psi: B_{h_1,\ell}\rightarrow MC(V,\ell)$ be given by $\Psi(v)= \psi(v)(1)$. Then $\Psi$ is smooth. In fact, there exists an open subset $0\in U\subset B_{h_1,\ell}$ such that $\Psi: U\subset \ker \ell_1 \rightarrow  MC(V,\ell)\subset \R^n$ is an embedding. 
\begin{proof}
By Theorem \ref{intpath} and the previous lemma we can compute
\[
	 \partial_{s=0} \Psi( s v)= \partial_{s=0} \Psi_h(v)(s)= v,
\]
and then $d_0 \Psi $ exists and is the identity, which implies the result.
\end{proof}
\end{thm}
\begin{coro}\label{MaxIntNew}
Let $(M\overset{\sigma}{\rightarrow} E\overset{\Phi}{\rightarrow} F)\curvearrowleft G$ be an analytic deformation problem with $(V,\ell)$ the associated $L_\infty$-structure described by Theorem \ref{arjen}. Assume that $(V,\ell)$ is $N$-strict. If $H^1(V,\ell)=0$, the problem is maximally integrable at $x_0$. In fact, around $x_0$, $\sigma^{-1}(0)$ has the smooth parametrization given by $\varphi^{-1}\circ \Psi: U\subset \ker\ell_1 \rightarrow M$. 
\begin{proof}
The map $\varphi^{-1}\circ\Psi: U\subset \ker\ell_1 \rightarrow \sigma^{-1}(0)\subset M$ gives an embedding.  The same argument as the one given at the end of the alternative proof of Theorem \ref{rigiditysurvey} proves that this embedding parametrizes, locally, all the possible zeros of $\sigma$.  
\end{proof}
\end{coro}

\begin{rmk} In fact, we can also obtain the results of this section when $(V, \ell)$ satisfies a weaker condition than $N$-strictness. It is enough that $\ell_{|\odot^k V_0}=0$ for all $k\geq N$. For example, this condition is satisfied for the $L_\infty$-algebras controlling simultaneous deformations of associative algebras and their morphisms (see \cite[Corollary 6.5]{FMY}), and of Lie algebras and their morphisms (see \cite[Corollary 8.5]{FMY} and \cite[Lemma 2.6]{FZ}).
\end{rmk}

\appendix
\section{Appendix: Combinatorics}

In this appendix we collect some results which are useful in dealing with the coefficients of the higher order derivatives of the Maurer-Cartan equation. 
\subsection{Super-Catalan numbers}
We will now study the sequence of numbers 
\[
	C_k=\sum_{i=2}^{k}\sum_{r_1+\cdots+r_i=k} C_{r_1}\cdots C_{r_i},
\]
with initial condition $C_1=1$. Let $S$ be a non-empty set.
\begin{dfn}
A \textit{word of length} $n$ in $S$ is a sequence $w=w_1\ldots w_n$, with $w_i\in S$. The \textit{alphabet generated by} $S$ is the free monoid generated by the words
\[
	Alf(S)=\langle w_1\ldots w_m\mid m\in\N \ \text{and} \ w_j\in S \rangle, 
\]
with multiplication given by concatenation 
\[
	v\cdot w:= v_1\ldots v_n w_1\ldots w_m,
\]
and identity given by the empty word. 
\end{dfn}
We call an element $w\in Alf(S)$ a \textit{word} and denote its length by $\grau{w}=n$. If $\grau{w}=1$ we call it a \textit{letter}. 
\begin{dfn}
A \textit{bracketing} of a word $w\in Alf(S)$ is obtained by expressing $w$ as the product of $2$ or more non-empty words $w=u_1\ldots u_k$, unless $w$ is a letter, and then inductively bracketing each $u_i$, until we get letters. 
\end{dfn}
\begin{exm}
Let $1 \ 2 \ 3 \in Alf(\N)$. The possible bracketings are given by 
\[
	(1)((2)(3)), \ ((1)(2))(3) \ \text{and} \ (1)(2)(3).
\]
\end{exm}
\begin{dfn}
The \textit{super-Catalan} $k$-number, denoted by $C_k$, is given by the number of different ways of bracketing a word $w\in Alf(S)$ with $\grau{w}=k$. 
\end{dfn}
By the previous example $C_3=3$. Indeed, the super-Catalan numbers satisfy the following recursive formula:
\begin{prop}\label{SCN}
Let $C_k$ be the super-Catalan $k$-number. Then 
\[
	C_{k}= \sum_{i=2}^k \sum_{r_1+\cdots+ r_i=k} C_{r_1} \cdots C_{r_i}. 
\]
\begin{proof}
The proof is by strong induction. Clearly the only bracketings of $1$ and $1 \ 2$ are $(1)$ and $(1)(2)$, i.e. $C_1=C_2=1$. Accordingly  
\[
	C_3= C_1C_2+C_2C_1+C_1C_1C_1= 3,
\]
and the base case follows. Suppose the result holds for $1\leq p\leq k-1$ and take the word $w\in Alf(\N)$ given by 
\[
	w= 1 \ 2 \ \cdots \ k. 
\]
Fix $2\leq i\leq k$ and a partition $r_1+\ldots +r_i=k$. Decompose $w$ by 
\[
	w= u_1 u_2 \cdots u_i,
\]
with $\grau{u_i}=r_i$. All the possible bracketings of this decomposition are given by the product of all the possible bracketings of each of the factors, i.e. the number of bracketings of this decomposition is $C_{r_1}\cdots C_{r_i}$. The number $i$ controls in how many different words can be decomposed $w$, i.e. from $2$ to $k$. The partitions $r_1+\ldots +r_i=k$ give all the possible decomposition into $i$ words of a word with $k$ letters. Since this covers all the possible decomposition of $w$ into smaller words, we get a bijection between the number of bracketings and the formula above.
\end{proof}
\end{prop}
The following asymptotic growth can be found at the Online Encyclopedia of Integer Sequences $(OEIS, A001003, 2025)$\cite{OeisSuperCatalan}.
\begin{prop}\label{asymp}
The super-Catalan numbers have the asymptotic growth
\[
	C_k\sim W\dfrac{(3+\sqrt{8})^k}{k^{3/2}},
\]
where $W=\dfrac{1+\sqrt{2}}{2^{7/4}\sqrt{\pi}}$.
\end{prop}

\subsection{Partitions of a number}\label{appendix-partition}
Let $k,i \in\N$ with $i<k$. 
\begin{dfn}
An $i$-\textit{partition of} $k$ is a sequence $(r_1,\ldots, r_i)\in \N^i$ such that $r_1+\cdots +r_i= k$. We denote it by $\vec{r}_i=k$ or $\vec{r}_i$ when there is no risk of confusion. An \textit{ordered} $i$-\textit{partition} of $k$ is a partition $\vec{r}_i$ such that $r_1\leq \cdots \leq r_i$.
\end{dfn}
\begin{dfn}
The \textit{multinomial coefficient of} $k$ \textit{associated with} $\vec{r}_i$ is defined to be 
 \[
 	\binom{k}{ \vec{r}_i}:= \binom{k}{r_1, \ldots, r_i}= \dfrac{k!}{r_1!\cdots r_i!}.
 \]
\end{dfn}
The permutation group $\sigma\in S_i$ acts on partitions $\vec{r}_{i}$ by 
\begin{equation*}\label{action}
(\sigma\cdot r)_j=r_{\sigma(j)}.
\end{equation*}
Two partitions that are on the same orbit are called \textit{equivalent}. Let $S_{i,{\vec{r}}}$ be the isotropy of the $S_i$ action on $\vec{r}_i$ and $S_i(\vec{r})$ its orbit. Denote by $\#$ the cardinality of a set. Then 
\begin{equation}\label{dimorb}
\# S_i (\vec{r})= \dfrac{i!}{\# S_{i,{\vec{r}}}}.
\end{equation}
We want an explicit description for $\# S_i(\vec{r})$.  To do so, notice that every partition $\vec{r}_{i}$ can be decomposed as 
\begin{equation}\label{factorization}
r_{a_1}+\overset{b_1}{\cdots}+r_{a_1}+\cdots+r_{a_{s}}+\overset{b_s}{\cdots}+r_{a_{s}}= k,
\end{equation}
with $r_{a_1}< \cdots< r_{a_s}$ and $b_1+\cdots+ b_s=i$. We call it the \textit{factorization} of  $\vec{r}_{i}$ \textit{into repeated factors} and denote it by $\vec{r}_{b_1, \ldots, b_s}$. The following lemma follows by construction. 
\begin{lemma}\label{repeatedfactor}
Let $\vec{r}_{i}$ be a $i$-partition of $k$ and $\vec{r}_{b_1, \ldots, b_s}$ its factorization into repeated factors. Then $\# S_{i,{\vec{r}}}= b_1! \cdots b_s!$.
\end{lemma}
By Equation \ref{dimorb} and Lemma \ref{repeatedfactor} we know that
\begin{equation}\label{dimorb2}
	\#S_i(\vec{r})= \dfrac{i!}{b_1! \cdots b_s!}.
\end{equation}

\section{Appendix: Higher order derivatives}

In this Appendix we prove a formula which is used in Proposition \ref{dMC1} to find obstructions in terms of higher order derivatives of the Maurer-Cartan equation; it is also used in Proposition \ref{aux} to prove that those obstructions are cohomological.

Let $(V,\ell)$ be an $L_\infty$-algebra with domain of convergence $D_\ell$ and take smooth paths $u_t,v_t: I\rightarrow D_\ell\subset V_0$. 
\begin{prop}\label{PropAuxB}
 For every $k\in\N$ the following formula holds
\begin{multline}\label{form1}
	\partial_t^{k} \ell_1^{u_t}(v_t)  \\
	=\ell_1^{u_t}(\partial_t^k v_t)+\sum_{i=1}^{k} \sum_{j=0}^{k-i} \sum_{\vec{r}_{i}=k-j }  \binom{k}{\vec{r}_{i}} \dfrac{1}{i!j!} \ \ell_{i+1}^{u_t}(\partial^{r_1}_t u_t\odot \cdots\odot \partial^{r_{i}}_t u_t\odot \partial^{j}_t v_t).
\end{multline}
\begin{proof}
The proof is by induction. The case $k=1$ follows from corollary \ref{derivative}. Assume that the result is true for $k$. Then 
\begin{multline}\label{dMC22}
	\partial_t^{k+1} \ell_1^{u_t}(v_t)= \partial_t \partial_t^k \ell_1^{u_t}(v_t) \\
	= \partial_t \ell_1^{u_t}(\partial_t^{k} v_t)+\sum_{i=1}^{k} \sum_{j=0}^{k-i} \sum_{\vec{r}_{i}=k-j }  \binom{k}{\vec{r}_{i}} \dfrac{1}{i!j!} \ \partial_t \ell_{i+1}^{u_t}(\partial^{r_1}_t u_t\odot \cdots\odot \partial^{r_{i}}_t u_t\odot \partial^{j}_t v_t).
\end{multline}
By Corollary \ref{derivative}, and after a careful verification, we have that \ref{dMC22} can be written as
\[
\ell_1^{u_t}(\partial_t^{k+1} v_t)+\sum_{m=1}^{k+1} \sum_{n=0}^{k+1-m} \sum_{\substack{\vec{p}_{m}=k+1-n \\ p_1\leq \ldots\leq p_m}} C_{\vec{p}_m,n} \ \ell_{m+1}^{u_t}(\partial^{p_1}_t u_t\odot \cdots\odot \partial^{p_{m}}_t u_t\odot \partial^{n}_t v_t),
\]
for some coefficients $C_{\vec{p}_m,n}$ to be determined. Notice that in the previous formula we are considering the sum over ordered partitions $\vec{p}_m$. On the other hand, by the symmetry of the $r_1,\ldots, r_{i}$ factors on the Equation \ref{form1}, we get that, for $k+1$, the Equation \ref{form1} is alternatively given by
\begin{multline*}\label{eqdef}
\partial_t^{k+1} \ell_1^{u_t} (v_t)= \ell_1^{u_t}(\partial_t^{k+1} v_t)\\
+\sum_{m=1}^{k+1} \sum_{n=0}^{k+1-m} \sum_{\substack{\vec{p}_{m}=k+1-n \\ p_1\leq \ldots\leq p_m}}  \binom{k+1}{\vec{p}_{m}} \dfrac{\# S_m(\vec{p})}{m!n!} \ \ell_{m+1}^{u_t}(\partial^{p_1}_t u_t\odot \cdots\odot \partial^{p_{m}}_t u_t\odot \partial^{n}_t v_t),
\end{multline*}
with $\# S_m(\vec{p})$ given by \ref{dimorb}. Accordingly, for the induction step we need to show that
\[
	C_{\vec{p}_m,n}= \binom{k+1}{\vec{p}_m} \dfrac{\# S_m(\vec{p})}{m!n!}.
\]
Thus, we need to find which summands of \ref{dMC22} contribute to each coefficient $C_{\vec{p}_m,n}$ and to determine that contribution. The strategy is the following: 
\begin{enumerate}
\item Fix $1\leq m\leq k+1$; $0\leq n\leq k+1-m$ and an ordered partition $p_1+\cdots +p_m=k+1-n$. 
\item Find all possible $1\leq i\leq k$, $0\leq j\leq k-i$ and ordered partitions $r_1+\cdots+r_i=k-j$ such that 
\begin{equation}\label{kira}
	\partial_t \ \ell_{i+1}^{u_t} (\partial^{r_1}_t u_t\odot\cdots\odot \partial^{r_i}_t u_t\odot \partial^j_t v_t)= \ell_{m+1}^{u_t} (\partial^{p_1}_t u_t\odot\cdots\odot \partial^{p_m}_t u_t\odot \partial^n_t v_t)+ \cdots .
\end{equation}
Denote by $S(m,n,\vec{p}_m)$ the set of all ordered partitions $\vec{r}_i=k-j$ satisfying \ref{kira}.
\item Take $\vec{r}_i\in S(m,n,\vec{p}_m)$. Use the Corollary \ref{derivative} to find the number of times that $\ell_{m+1}^{u_t} (\partial^{p_1}_t u_t\odot\cdots\odot \partial^{p_m}_t u_t\odot 				\partial^n_t v_t)$ appears in the expression $\partial_t \ \ell_{i+1}^{u_t} (\partial^{r_1}_t u_t\odot\cdots\odot \partial^{r_i}_t u_t\odot \partial^j_t v_t)$. Explicitly, if we denote this number by $A_{\vec{r}_i}$, we have that
\begin{multline}\label{Step3}
		\partial_t \ \ell_{i+1}^{u_t} (\partial^{r_1}_t u_t\odot\cdots\odot \partial^{r_i}_t u_t\odot \partial^j_t v_t) \\ =A_{\vec{r}_i} \cdot \ell_{m+1}^{u_t} (\partial^{p_1}_t u_t\odot\cdots\odot \partial^{p_m}_t u_t\odot 				\partial^n_t v_t)+ \cdots.
\end{multline}	
On the other hand, by the Equation \ref{dMC22}, the coefficient of the partition $\vec{r}_i$ in \ref{dMC22} is $\dfrac{1}{i!j!}\cdot \binom{k}{\vec{r}_i}.$ Hence, letting $C_{\vec{r}_i,j}$ be the contribution of $\vec{r}_i$ to $C_{\vec{p}_m,n}$, we conclude that
	\[
		C_{\vec{r}_i,j}= \dfrac{A_{\vec{r}_i}}{i!j!}\cdot \binom{k}{\vec{r}_i}.
	\]
\item Notice that if $\vec{r}_i\in S(m,n,\vec{p}_m)$ then $\vec{r}_i\cdot \sigma$ also contributes  to $C_{\vec{p}_m,n}$ for all $\sigma\in S_i$. Moreover, $A_{\vec{r}_i\cdot\sigma}= A_{\vec{r}_i}$ and the coefficients of the partitions $\vec{r}_i\cdot \sigma$ in \ref{dMC22} are all equal. Hence, the contribution of all the possible partitions $\vec{r}_i\cdot \sigma$  to $C_{\vec{p}_m,n}$ is given by 
\[
	\#\vec{r}_i:=\sum_{\sigma\in S_i} C_{\vec{r}_i\cdot\sigma,j}= \# S_i(\vec{r})\cdot C_{\vec{r}_i,j}.
\]
Accordingly, we need to find $ \# S_i(\vec{r})$.
\item Finally, we need to show that 
\[
	\sum_{(i,j,\vec{r}_i)\in S(m,n,\vec{p}_m)} \#\vec{r}_i= C_{\vec{p}_m,n}= \binom{k+1}{\vec{p}_m} \dfrac{\# S_m(\vec{p})}{m!n!}.
\]
\end{enumerate}
Let us now elaborate each of the previous steps:
\begin{enumerate}
	\item Take $m,n,\vec{p}_m$ and let 
		\[
			b_1 p_{a_1}+ \cdots+ b_s p_{a_s} = k+1-n,
		\]
 	be its factorization into repeated factors (see \ref{factorization}). Thus, by the Equation \ref{dimorb2}, we conclude that
		\begin{equation}\label{auxcoeff}
			\binom{k+1}{\vec{p}_m} \dfrac{\# S_m(\vec{p})}{m!n!}= \dfrac{1}{b_1!\cdots b_s! n! } \binom{k+1}{\vec{p}_m}.
		\end{equation}
	\item By the Corollary \ref{derivative}, we have four possible situations for $\vec{r}_i=k-j$ satisfying \ref{kira}:
		\begin{enumerate}
			\item $p_{a_1}=1$ and $\vec{r}_i$ is obtained by subtracting 1 to $p_1$:
				\[
					\vec{r}_{m-1}= (p_2, \ldots, p_m),
				\]
			so $(m-1,n,\vec{r}_{m-1})\in S(m,n,\vec{p}_m)$. 
			\item $p_{a_1}>1$ and $\vec{r}_i$ is obtained by subtracting 1 to $p_1$:
				\[
					\vec{r}_{m}=(p_{1}-1, p_2, \ldots, p_m),
				\]
			so $(m, n, \vec{r}_{m})\in S(m,n,\vec{p}_m)$. 
			\item $\vec{r}_i$ is obtained by subtracting 1 to $p_{a_d}$, for each $2\leq d\leq s$: define $1\leq q(d)\leq m$ by
				\[
					 q(d)= b_1+\cdots+ b_{d-1}+1.
				\] 
			We need to consider all the
				\[
					\vec{r}_m= (p_1,\ldots, p_{q(d)}-1, \ldots, p_m),
				\]
			so $(m,n,\vec{r}_m)\in S(m,n,\vec{p}_m)$.
			\item $j=n-1$, forcing 
				\[
					\vec{r}_m= \vec{p}_m,
				\]
			so $(m,n-1,\vec{r}_m)\in S(m,n,\vec{p}_m)$.
		\end{enumerate}
	These four cases cover all possible elements $(i,j,\vec{r}_i)\in S(m,n,\vec{p}_m)$. 
	\item Recall that $A_{\vec{r}_{i}}$ is defined by Equation \ref{Step3}. Using Corollary \ref{derivative} it follows that $A_{\vec{r}_{i}}$, in each of the previous four cases, is given by:
		\begin{enumerate}
			\item $A_{\vec{r}_i}=1$ and then 
				\[
					C_{\vec{r}_i,j}= \dfrac{1}{(m-1)!n!} \binom{k}{p_2,\cdots, p_m}.
				\]
			\item $A_{\vec{r}_i}=1$ and then 
				\[
					C_{\vec{r}_i,j}=\dfrac{1}{m!n!}\binom{k}{p_1-1,p_2,\cdots, p_m}.	
				\] 
			\item We have two cases:
				\begin{enumerate}[i)]
					\item If $p_{a_{q(d)-1}}+1= p_{a_{q(d)}}$ then $A_{\vec{r}_i}= b_{q(d)-1}+1$. Thus
						\[
							C_{\vec{r}_i,j}= \dfrac{b_{q(d)-1}+1}{m!n!}\binom{k}{p_1,\cdots, p_{a_{q(d)}}-1, \cdots, p_m}.
						\]
					\item If $p_{a_{q(d)-1}}+1< p_{a_{q(d)}}$ then $A_{\vec{r}_i}=1$. Thus 
						\[
							C_{\vec{r}_i,j}= \dfrac{1}{m!n!}\binom{k}{p_1,\cdots, p_{a_{q(d)}}-1, \cdots, p_m}.
						\]
				\end{enumerate}
			\item $A_{\vec{r}_i}=1$ and then 
				\[
					C_{\vec{r}_i,j}=\dfrac{1}{m!(n-1)!} \binom{k}{p_1,\cdots, p_m}. 
				\]
		\end{enumerate}
	\item Using the equation \ref{dimorb2}, with respect to the previous cases, we have that:
		\begin{enumerate}
			\item 
				\[
					\#S_i(\vec{r})=\dfrac{(m-1)!}{(b_1-1)!b_2!\cdots b_s!},
				\]
			and then  
				\[
					\# \vec{r}_i= \dfrac{1}{(b_1-1)!b_2!\cdots b_s!n!} \binom{k}{p_2,\cdots, p_m}.
				\]
			\item 
				\[
					\#S_i(\vec{r})=\dfrac{m!}{(b_1-1)!b_2!\cdots b_s!},
				\]
			and then  
				\[
					\# \vec{r}_i= \dfrac{1}{(b_1-1)!b_2!\cdots b_s!n!} \binom{k}{p_1-1,\cdots, p_m}.
				\]
			\item 
				\begin{enumerate}[i)]
					\item 
						\[
							\#S_i(\vec{r})=\dfrac{m!}{b_1\cdots (b_{q(d)-1}+1)! (b_{q(d)}-1)!\cdots b_s!}.
						\]
					\item 
						\[
							\#S_i(\vec{r})=\dfrac{m!}{b_1\cdots  (b_{q(d)}-1)!\cdots b_s!}.
						\]
				\end{enumerate}
			In both cases we get that   
				\[
					\# \vec{r}_i= \dfrac{1}{b_1!\cdots (b_{q(d)}-1)!\cdots b_s!n!}\binom{k}{p_1,\cdots, p_{a_{q(d)}}-1, \cdots, p_m}
				\]
			\item 
				\[
					\#S_i(\vec{r})=\dfrac{1}{m!(n-1)!} \binom{k}{p_1,\cdots, p_m},
				\]
			and then 
				\[
					\# \vec{r}_i= \dfrac{1}{b_1!\cdots b_s!(n-1)!} \binom{k}{p_1,\cdots, p_m}.
				\]
		\end{enumerate}
	\item We are going to compute 
		\[
			\sum_{(i,j,\vec{r}_i)\in S(m,n,\vec{p}_m)} \#\vec{r}_i,
		\]
	where $\vec{p}_m= b_1 p_{a_1}+\cdots+ b_s p_{a_s}=k+1-n$, with $p_{a_1}>1$ and $n>0$. The other cases follow by similar computations. By $2.$ the possible triples $(\vec{r}_i, i,j)\in S(\vec{p}_m,m,n)$ are the cases $b)$, $c)$ and $d)$. Hence 
	\[
		\begin{split}
			&\sum_{(i,j,\vec{r}_i)\in S(m,n,\vec{p}_m)} \#\vec{r}_i= \dfrac{1}{(b_1-1)!b_2!\cdots b_s!n!} \binom{k}{p_1-1,\cdots, p_m} \\
			&+ \sum_{d=2}^s \dfrac{1}{b_1!\cdots (b_{q(d)}-1)!\cdots b_s!n!}\binom{k}{p_1,\cdots, p_{a_{q(d)}}-1, \cdots, p_m} \\
			&+  \dfrac{1}{b_1!\cdots b_s!(n-1)!} \binom{k}{p_1,\cdots, p_m} \\
			&\overset{(*)}{=}\dfrac{1}{b_1!\cdots b_s!n!}\left(\sum_{q=1}^m \binom{k}{p_1,\cdots, p_q-1, \cdots, p_m}+ n \binom{k}{p_1,\cdots, p_m} \right)\\
			&= \dfrac{1}{b_1!\cdots b_s!n!}\binom{k+1}{p_1,\cdots, p_m}.
		\end{split}
	\] Equality $(*)$ is obtained using that \[\begin{split}b_1\binom{k}{p_1-1, \cdots, p_m}&=b_1\binom{k}{p_{a_1}-1,\stackrel{b_1}{\cdots},p_{a_1}, \cdots,p_{a_s},\stackrel{b_s}{\cdots}, p_{a_s}}\\&=\sum_{q=1}^{b_1}\binom{k}{p_1,\cdots,p_{q}-1,\cdots, p_m},\end{split}\]	
	and similar arguments for $b_{q(d)}$. By Equation \ref{auxcoeff} it follows that 
		\[
			\sum_{(i,j,\vec{r}_i)\in S(m,n,\vec{p}_m)} \#\vec{r}_i= \binom{k+1}{\vec{p}_m} \dfrac{\# S_m(\vec{p})}{m!n!}= C_{\vec{p}_m,n}.\qedhere
		\]
\end{enumerate}
\end{proof}
\end{prop}

\addcontentsline{toc}{section}{References} 
\addtocontents{toc}{\vspace{2em}}
\bibliographystyle{alpha}
\bibliography{references}

\end{document}